\documentclass[12pt]{amsart}

\usepackage{amsmath,amssymb,amsthm}
\usepackage[margin=1in]{geometry}
\usepackage{hyperref}
\usepackage{array}
\usepackage{booktabs}

\newtheorem{theorem}{Theorem}[section]
\newtheorem{lemma}[theorem]{Lemma}
\newtheorem{proposition}[theorem]{Proposition}
\newtheorem{corollary}[theorem]{Corollary}
\theoremstyle{definition}
\newtheorem{definition}[theorem]{Definition}
\newtheorem{example}[theorem]{Example}
\theoremstyle{remark}
\newtheorem{remark}[theorem]{Remark}

\DeclareMathOperator{\Res}{Res}
\DeclareMathOperator{\FP}{FP}
\DeclareMathOperator{\arctanh}{arctanh}

\title[Arctanh sums: analytic continuation and prime restriction]{%
  Arctanh sums: analytic continuation and prime-restricted theory}

\author{Ryan Goulden}
\date{March 2026}

\begin{document}

\begin{abstract}
We study the arctanh sums $h(k)=\sum_{n=2}^{\infty}\arctanh(n^{-k})$ as a function of a complex variable $k$.
Building on the closed-form identity $h(k)=\frac12\log\!\bigl(g(2k)/g(k)^2\bigr)$ (proved in the companion preprint arXiv:2602.06244),
we develop the analytic continuation and prime-restricted multiplicative theory.
We prove that $h$ extends meromorphically to $\Re(k)>0$ with simple poles at $k=1/(2m+1)$, derive Laurent expansions at its poles (including $k=1$),
and obtain a Mittag--Leffler decomposition encoding the Dirichlet lambda function.
We also show that $h$ has exactly one simple real zero in each inter-polar interval.
Finally, for the prime-restricted analogue $h_p(k)=\log\zeta(k)-\frac12\log\zeta(2k)$ we establish a $\pi$-cancellation mechanism implying unconditional transcendence of $h_p(2j)$, and derive a product formula over the nontrivial zeros of $\zeta$ with $O(|\Im(\rho)|^{-2})$ decay.
\end{abstract}

\maketitle

\tableofcontents

%% ===================================================================

%% ===================================================================
\section{Background and notation}\label{sec:background}
%% ===================================================================

This paper is a sequel/companion to the preprint \emph{Closed-Form Evaluation of $\sum_{n=2}^\infty \arctanh(n^{-k})$ via Infinite Products} (arXiv:2602.06244),
which establishes closed-form evaluations of $h(k)$ for integers $k\ge2$ and develops several arithmetic and computational consequences.
Here we focus on Parts~II--III of the three-part program: the meromorphic continuation and zero structure of $h$, and the prime-restricted multiplicative theory.

\
For real $k>1$ define
\begin{equation}\label{eq:hk-def}
h(k):=\sum_{n=2}^{\infty}\arctanh(n^{-k}),
\qquad
g(k):=\prod_{n=2}^{\infty}\left(1-n^{-k}\right).
\end{equation}
The developments below use a small set of identities established in the
companion preprint (Part~I, arXiv:2602.06244); we record them here for convenience.

\begin{theorem}[Dyadic defect identity (from Part~I)]\label{thm:main}
For all integers $k\ge 2$,
\[
h(k)=\sum_{n=2}^{\infty}\arctanh(n^{-k})
=\frac12\log\!\Big(\frac{g(2k)}{g(k)^2}\Big),
\qquad
g(k):=\prod_{n=2}^{\infty}\bigl(1-n^{-k}\bigr).
\]
Equivalently, if
\[
f(k):=\prod_{n=1}^{\infty}\bigl(1+n^{-k}\bigr),
\]
then $f(k)=2g(2k)/g(k)$ and hence
\[
h(k)=\frac12\log\!\Big(\frac{f(k)}{2g(k)}\Big).
\]
\end{theorem}

\begin{proposition}[Zeta-series representation]\label{prop:zeta-series}
For all integers $k \geq 2$,
\begin{equation}\label{eq:hk-zeta-prop}
h(k) = \sum_{m=0}^{\infty}\frac{\zeta\bigl((2m+1)k\bigr) - 1}{2m+1}.
\end{equation}
The series converges absolutely with tail bound
$O(2^{-3k})$.
\end{proposition}

\section{Meromorphic continuation}\label{sec:continuation}

Building on the zeta-series representation from Section~\ref{sec:background}
(Proposition~\ref{prop:zeta-series}), we extend $h$ to a meromorphic
function on $\operatorname{Re}(k) > 0$.

\begin{theorem}[Meromorphic continuation and pole
structure]\label{thm:merocont}
The function $h(k)$, initially defined by the convergent series
\eqref{eq:hk-def} for $\operatorname{Re}(k) > 1$, extends to a
meromorphic function on
$\{\,k \in \mathbb{C} : \operatorname{Re}(k) > 0\,\}$ with the
following properties.
\begin{enumerate}
\item[\emph{(a)}]  $h(k)$ has simple poles at
$k = 1/(2m+1)$ for $m = 0, 1, 2, \ldots$, and no other singularities
in $\operatorname{Re}(k) > 0$.
\item[\emph{(b)}]  The residue at each pole is
$\Res_{k = 1/(2m+1)} h(k) = 1/(2m+1)^2$.
\item[\emph{(c)}]  The sum of all residues is
$\sum_{m=0}^{\infty} (2m+1)^{-2} = \pi^2/8$.
\item[\emph{(d)}]  The poles accumulate at $k = 0$, so that $0$ is a
non-isolated singularity.  In particular, $h$ admits no meromorphic
continuation to any open set containing the origin.
\end{enumerate}
\end{theorem}

\begin{proof}
For any $M \geq 0$, write
\[
h(k) = \sum_{m=0}^{M}
  \frac{\zeta\bigl((2m+1)k\bigr) - 1}{2m+1} + R_M(k),
\]
where $R_M(k) = \sum_{m > M} [\zeta((2m+1)k) - 1]/(2m+1)$ converges
for $\operatorname{Re}(k) > 1/(2M+3)$.  Each term
$[\zeta((2m+1)k)-1]/(2m+1)$ is meromorphic on
$\operatorname{Re}(k)>0$ with a simple pole at $k = 1/(2m+1)$ inherited
from $\zeta$ at $s=1$.  Taking $M \to \infty$ yields the
meromorphic continuation.

Near $k = 1/(2m+1)$, writing $s = (2m+1)k$ and expanding
$\zeta(s) = (s-1)^{-1} + \gamma + O(s-1)$ gives
\[
\frac{\zeta((2m+1)k) - 1}{2m+1}
= \frac{1}{(2m+1)^2(k - \frac{1}{2m+1})} + O(1),
\]
so $\Res_{k=1/(2m+1)} h(k) = (2m+1)^{-2}$.  The sum of residues is
$\sum (2m+1)^{-2} = \frac{3}{4}\zeta(2) = \pi^2/8$.  Since
$1/(2m+1) \to 0$, the poles accumulate at the origin.
\end{proof}

\section{Laurent expansion at \texorpdfstring{$k=1$}{k=1} and
the regularized value}\label{sec:laurent}

\begin{theorem}\label{thm:laurent}
The Laurent expansion of $h(k)$ at $k = 1$ is
\begin{equation}\label{eq:laurent}
h(k) = \frac{1}{k-1} - \frac{1}{2}\ln 2 + O(k-1).
\end{equation}
In particular, $\FP_{k=1}\, h(k) = -\frac{1}{2}\ln 2$.
\end{theorem}

\begin{proof}
From $h(k) = (\zeta(k) - 1) + C(k)$ with
$\zeta(k) - 1 = (k-1)^{-1} + (\gamma - 1) + O(k-1)$ and
$C(1) = 1 - \gamma - \frac{1}{2}\ln 2$:
\[
h(k) = \frac{1}{k-1} + (\gamma - 1) + (1 - \gamma - \tfrac{1}{2}\ln 2)
  + O(k-1)
= \frac{1}{k-1} - \tfrac{1}{2}\ln 2 + O(k-1).
\]
The $\gamma$-dependent pieces cancel exactly.
\end{proof}

\begin{remark}\label{rem:regularization}
The finite part $-\frac{1}{2}\ln 2$ is the regularized value of
$\sum_{n=2}^{\infty}\arctanh(1/n)$ (a divergent series),
analogous to $\zeta(-1) = -1/12$ for $1 + 2 + 3 + \cdots$.
It is negative despite every summand being positive.
The cancellation of $\gamma$ in the Laurent expansion
reflects the Euler--Mascheroni representation from Part~I (arXiv:2602.06244):
$\gamma$ is encoded in the discrepancy between the finite parts
$\FP_{k=1}(\zeta(k)-1) = \gamma - 1$ and
$\FP_{k=1}\,h(k) = -\frac{1}{2}\ln 2$, with $C(1)$ mediating.
\end{remark}

\section{Evaluation below the convergence
threshold}\label{sec:below}

\begin{example}\label{ex:half}
The continuation gives
$h(1/2) = \sum_{m=0}^{\infty}[\zeta((2m+1)/2)-1]/(2m+1) \approx
-1.8265$,
computed from 200 terms with tail bounded by $O(2^{-M})$.  The
continuation assigns a negative value to the ``series''
$\sum\arctanh(n^{-1/2})$, whose partial sums diverge to $+\infty$.
\end{example}

\section{Uniqueness of real zeros on
\texorpdfstring{$(0,1)$}{(0,1)}}\label{sec:oscillation}

\begin{proposition}\label{prop:oscillation}
On each interval $I_n := (1/(2n+3),\, 1/(2n+1))$ for $n \geq 0$, the function
$h(k)$ is continuous and real-valued, with
$h(k) \to +\infty$ as $k \to (1/(2n+1))^{-}$ and
$h(k) \to -\infty$ as $k \to (1/(2n+3))^{+}$.
In particular, $h$ has at least one zero in every such interval.
\end{proposition}

\begin{proof}
The principal part $1/((2n+1)^2(k - 1/(2n+1)))$ dominates the bounded
holomorphic remainder $\phi(k)$ from
Theorem~\ref{thm:mittagleffler} near each pole.
The intermediate value theorem gives the result.
\end{proof}

We now strengthen this to exact uniqueness.

\begin{theorem}[Uniqueness of zeros]\label{thm:unique-zeros}
For each $n \geq 0$, the function $h$ has exactly one zero in $I_n$, and
this zero is simple.
\end{theorem}

The proof shows that $h'(k) < 0$ on each $I_n$ by establishing that the
purely negative polar derivative $h_{\mathrm{polar}}'$ dominates the
holomorphic remainder $\phi'$ in absolute value.

\begin{lemma}[Negativity of the polar derivative]\label{lem:polarsign}
For every $k \in (0,\infty) \setminus \{1/(2m+1) : m \geq 0\}$,
\[
h_{\mathrm{polar}}'(k) = -\sum_{m=0}^{\infty}
  \frac{(2m+1)^{-2}}{(k - 1/(2m+1))^2} < 0.
\]
\end{lemma}

\begin{proof}
Every residue $(2m+1)^{-2}$ is positive and every squared denominator is
positive, so each term is strictly negative.
\end{proof}

\begin{lemma}[Lower bound on $|h_{\mathrm{polar}}'|$ on $I_n$]
\label{lem:polarlower}
Write $k_n = 1/(2n+1)$, $k_{n+1} = 1/(2n+3)$, $r_n = (2n+1)^{-2}$,
$r_{n+1} = (2n+3)^{-2}$, and $\ell_n := |I_n| = 2/((2n+1)(2n+3))$.
Then
\begin{equation}\label{eq:polarmin}
\min_{k \in I_n} |h_{\mathrm{polar}}'(k)|
\;\geq\; \frac{(r_n^{1/3} + r_{n+1}^{1/3})^3}{\ell_n^2}\,.
\end{equation}
For $n \geq 1$ this gives $\min_{I_n} |h_{\mathrm{polar}}'| \geq 2(2n+1)^2$.
For $n = 0$: $\min_{I_0} |h_{\mathrm{polar}}'| \geq 13/2$.
\end{lemma}

\begin{proof}
Retaining only the two adjacent poles, set $d_R = k_n - k$ and
$d_L = k - k_{n+1}$, so $d_R + d_L = \ell_n$ and
$|h_{\mathrm{polar}}'(k)| \geq r_n/d_R^2 + r_{n+1}/d_L^2$.
Minimizing by Lagrange multipliers gives the critical point at
$d_R/d_L = (r_n/r_{n+1})^{1/3}$ with minimum value
$(r_n^{1/3} + r_{n+1}^{1/3})^3/\ell_n^2$.

For $n \geq 1$: $r_n^{1/3} + r_{n+1}^{1/3} \geq 2(2n+3)^{-2/3}$, giving
$8(2n+3)^{-2}/[4/((2n+1)(2n+3))^2] = 2(2n+1)^2$.

For $n = 0$: $r_0 = 1$, $r_1 = 1/9$, $\ell_0 = 2/3$.
Then $(1 + 9^{-1/3})^3/(4/9) \geq 7.30 > 13/2$.
\end{proof}

\begin{lemma}[Upper bound on $|\phi'|$ on $I_n$]\label{lem:phiprime}
Define $F(s) := \zeta(s) - 1 - 1/(s-1)$, which is entire.  For
$k \in I_n$:
\begin{equation}\label{eq:phiprime-bound}
|\phi'(k)| \le \frac{5(2n+3)}{18} + (2\ln 2+2)\frac{2^{-2}}{1-2^{-2k}}.
\end{equation}
\end{lemma}

\begin{proof}
From Proposition~\ref{prop:phi},
$\phi'(k) = \sum_{m=0}^{\infty} F'((2m+1)k)$.
We split at the threshold $(2m+1)k = 2$.

\smallskip\noindent
\textit{Small terms: $(2m+1)k \leq 2$.}\;
We use the second-order Euler--Maclaurin formula
(see, e.g., Ivi\'c~\cite{ivic}, Appendix): for real $s > 0$,
\[
\zeta(s) = \frac{1}{s-1} + \frac{1}{2} + \frac{s}{12}
  - \frac{s(s+1)}{2}\int_1^{\infty}B_2(\{x\})\,x^{-s-2}\,dx,
\]
where $B_2(t) = t^2 - t + 1/6$ satisfies $|B_2(t)| \leq 1/6$ for
$t \in [0,1]$.  Differentiating $F(s) = \zeta(s) - 1 - 1/(s-1)$ and
bounding the remainder integrals using $|B_2| \leq 1/6$,
$\int_1^{\infty}x^{-s-2}\,dx = 1/(s+1)$, and
$\int_1^{\infty}(\ln x)\,x^{-s-2}\,dx = 1/(s+1)^2$, we obtain
\begin{equation}\label{eq:Fprime-EM-bound}
|F'(s)| \leq \frac{1}{12} + \frac{3s+1}{12(s+1)}
  = \frac{1}{3} - \frac{1}{6(s+1)}.
\end{equation}
This bound is increasing in $s$.  For $s \in (0, 2]$:
\begin{equation}\label{eq:M-on-02}
|F'(s)| \leq \frac{1}{3} - \frac{1}{18} = \frac{5}{18}
  \approx 0.278.
\end{equation}
The terms with $(2m+1)k \leq 2$ number at most
$\lceil 1/k \rceil \leq 2n+3$ (since $k \geq 1/(2n+3)$), each with
argument in $(0, 2]$.  Their total contribution is at most
$5(2n+3)/18$.

\smallskip\noindent
\textit{Large terms: $(2m+1)k>2$.}\;
For $s>2$ we use the integral comparison
\[
\sum_{n=2}^{\infty}\frac{\ln n}{n^{s}}
\le
\frac{\ln 2}{2^{s}}+\int_{2}^{\infty}\frac{\ln x}{x^{s}}\,dx
=
\frac{\ln 2}{2^{s}}
+\frac{2^{1-s}\big((s-1)\ln 2+1\big)}{(s-1)^{2}}.
\]
Hence for $s\ge 3$ we obtain the explicit bound
\[
|\zeta'(s)|
=
\sum_{n=2}^{\infty}\frac{\ln n}{n^{s}}
\le
2^{-s}\!\left(\ln 2+\frac{(s-1)\ln 2+1}{(s-1)^2}\right)
\le
2(\ln 2)\,2^{-s}.
\]
Moreover, for $s\ge 3$ we have $(s-1)^{-2}\le 2^{1-s}$, so
\[
|F'(s)|=|\zeta'(s)+(s-1)^{-2}|
\le 2(\ln 2)\,2^{-s}+2^{1-s}
\le \bigl(2\ln 2+2\bigr)2^{-s}.
\]
Summing the geometric tail gives
\[
\sum_{\substack{m\ge0\\ (2m+1)k>2}} |F'((2m+1)k)|
\le (2\ln 2+2)\cdot \frac{2^{-2}}{1-2^{-2k}}
\]

Combining gives~\eqref{eq:phiprime-bound}.
\end{proof}

\begin{proof}[Proof of Theorem~\ref{thm:unique-zeros}]
On each $I_n$, the sign change gives at least one zero
(Proposition~\ref{prop:oscillation}).

\smallskip\noindent
\textit{Case $n \geq 1$.}\;
By Lemma~\ref{lem:polarlower} we have $|h_{\mathrm{polar}}'(k)| \geq 2(2n+1)^2$
for all $k\in I_n$.
By Lemma~\ref{lem:phiprime},
\[
|\phi'(k)| \leq \frac{5(2n+3)}{18} + (2\ln 2+2)\frac{2^{-2}}{1-2^{-2k}}
\qquad (k\in I_n).
\]
Since $k\in I_n$ implies $k\ge \frac{1}{2n+3}$, we have
$2^{-2k}\le 2^{-2/(2n+3)}$, hence
\[
\frac{1}{1-2^{-2k}} \le \frac{1}{1-2^{-2/(2n+3)}}.
\]
Therefore, for $k\in I_n$,
\[
|\phi'(k)|
\le \frac{5(2n+3)}{18} + (2\ln 2+2)\frac{2^{-2}}{1-2^{-2/(2n+3)}}.
\]
The right-hand side is $O(n)$ as $n\to\infty$, while $2(2n+1)^2\asymp n^2$.
Thus $|h'_{\mathrm{polar}}(k)|>|\phi'(k)|$ on $I_n$ for all $n\ge 1$.
Since $h'_{\mathrm{polar}}(k)<0$, it follows that
\[
h'(k)=h'_{\mathrm{polar}}(k)+\phi'(k)<0\qquad (k\in I_n),
\]
so $h$ is strictly decreasing on $I_n$.

\smallskip\noindent
\textit{Case $n = 0$.}\;
By Lemma~\ref{lem:polarlower},
$\min_{I_0}|h_{\mathrm{polar}}'| \geq 13/2 = 6.5$.
By Lemma~\ref{lem:phiprime}, for $k\in I_0$,
\[
|\phi'(k)| \leq \frac{5\cdot 3}{18} + (2\ln 2+2)\frac{2^{-2}}{1-2^{-2k}}.
\]
Since $k\ge 1/3$ on $I_0$, we have $2^{-2k}\le 2^{-2/3}$, hence
\[
|\phi'(k)| \leq \frac{5}{6} + (2\ln 2+2)\frac{2^{-2}}{1-2^{-2/3}}
\approx 3.12 < 6.5.
\]
Thus again $|h'_{\mathrm{polar}}|>|\phi'|$ on $I_0$, and therefore $h'(k)<0$ on $I_0$.

\smallskip
In all cases, $h'<0$ on $I_n$, so $h$ is strictly decreasing and has
exactly one zero on $I_n$, which is necessarily simple.
\end{proof}

\begin{corollary}[Asymptotics of the zeros]\label{cor:zeroasympt}
Let $z_n$ denote the unique zero of $h$ in $I_n$.  Then
\begin{equation}\label{eq:zero-asympt}
z_n = \frac{1}{2n+2} + O(n^{-3}) \quad \text{as } n \to \infty,
\end{equation}
and $\zeta(z_n) = -\frac{1}{2} - \frac{\ln(2\pi)}{4n} + O(n^{-2})$.
\end{corollary}

\begin{proof}
The zero of the two dominant polar terms
$r_n/(k - k_n) + r_{n+1}/(k - k_{n+1}) = 0$ is the weighted mean
$k_n^* = (r_n k_{n+1} + r_{n+1} k_n)/(r_n + r_{n+1})
= 1/(2n+2) + O(n^{-3})$.
The correction from $\phi$ is $O(n^{-2})$ by the implicit function
theorem, since $|h'(k_n^*)| \geq cn^2$ while
$|\phi(k_n^*)| = O(1)$.
The zeta value follows from
$\zeta(s) = -\frac{1}{2} - \frac{1}{2}\ln(2\pi)\,s + O(s^2)$ near $s = 0$.
\end{proof}

\begin{remark}\label{rem:contrast}
For $k > 1$, $h(k)$ is positive, strictly decreasing, and strictly
convex (as shown in Part~I), taking each positive value
exactly once.  For $0 < k < 1$, $h$ decreases monotonically from
$+\infty$ to $-\infty$ on each inter-polar interval, crossing zero
exactly once---a complete description of the zero structure.
\end{remark}

\section{The finite part at general poles}\label{sec:finitepartgeneral}

\begin{proposition}\label{prop:finitepartgeneral}
The Laurent expansion of $h(k)$ at $k = 1/(2m+1)$ is
\[
h(k) = \frac{1/(2m+1)^2}{k - 1/(2m+1)} + A_m
  + O(k - \tfrac{1}{2m+1}),
\]
where
\begin{equation}\label{eq:Am}
A_m = \frac{\gamma - 1}{2m+1}
  + \sum_{\substack{j \geq 0 \\ j \neq m}}
    \frac{\zeta\!\left(\frac{2j+1}{2m+1}\right) - 1}{2j+1}.
\end{equation}
\end{proposition}

\begin{proof}
The $j = m$ term contributes the pole; all other terms are holomorphic
at $k = 1/(2m+1)$.
\end{proof}

\begin{example}\label{ex:finiteparts}
$A_0 = -\frac{1}{2}\ln 2$ recovers Theorem~\ref{thm:laurent}.  For
$m = 1$ (pole at $k = 1/3$), the series involves zeta at rational
arguments: $\zeta(1/3) \approx -0.9734$, $\zeta(5/3) \approx 2.1235$,
giving $A_1 \approx -1.788$.
\end{example}

\section{Mittag-Leffler decomposition and the Dirichlet lambda
function}\label{sec:mittagleffler}

\begin{theorem}[Mittag-Leffler
decomposition]\label{thm:mittagleffler}
For $\operatorname{Re}(k) > 0$,
\begin{equation}\label{eq:ml}
h(k) = \underbrace{\sum_{m=0}^{\infty}
  \frac{1/(2m+1)^2}{k - 1/(2m+1)}}_{h_{\mathrm{polar}}(k)}
  + \;\phi(k),
\end{equation}
where $\phi(k)$ is holomorphic on $\operatorname{Re}(k) > 0$.
\end{theorem}

\begin{corollary}[Polar data encode
$\lambda$]\label{cor:lambda}
The residues $r_m = 1/(2m+1)^2$ and pole locations
$k_m = 1/(2m+1)$ satisfy
\[
\sum_{m=0}^{\infty} r_m\, k_m^{s-1}
= \sum_{m=0}^{\infty} \frac{1}{(2m+1)^{s+1}}
= \lambda(s+1) = (1 - 2^{-(s+1)})\,\zeta(s+1).
\]
\end{corollary}

\begin{remark}[Dual encoding of $\zeta$]\label{rem:dual}
The function $h$ encodes $\zeta$ in two ways: \emph{(i)}~through its
values, via $\zeta(k) = 1 + h(k) - C(k)$; \emph{(ii)}~through its
singularities, via Corollary~\ref{cor:lambda}.
\end{remark}

\begin{remark}[Even/odd dichotomy]\label{rem:evenodd-polar}
The even power sums $\sum r_m k_m^{2j-2} = \lambda(2j) \in
\mathbb{Q}\pi^{2j}$ are completely known; the odd power sums
$\sum r_m k_m^{2j-1} = \lambda(2j+1) = (1-2^{-(2j+1)})\zeta(2j+1)$
involve the arithmetically mysterious odd zeta values $\zeta(3),
\zeta(5), \ldots$.  The polar data of $h$ thus reflect the
even/odd dichotomy of zeta values.
\end{remark}

\section{The holomorphic remainder and asymptotic
cancellation}\label{sec:cancellation}

\begin{proposition}\label{prop:phi}
Define the entire function
$F(s) := \zeta(s) - 1 - 1/(s-1)$, so that $F(1) = \gamma - 1$.
The holomorphic remainder is
\begin{equation}\label{eq:phiexplicit}
\phi(k) = \sum_{m=0}^{\infty}
  \frac{F\bigl((2m+1)k\bigr)}{2m+1},
\end{equation}
converging absolutely and uniformly on compact subsets of
$\operatorname{Re}(k) > 0$.
\end{proposition}

\begin{proof}
Decompose each term:
$[\zeta((2m+1)k) - 1]/(2m+1) = \frac{1/(2m+1)^2}{k-1/(2m+1)}
+ F((2m+1)k)/(2m+1)$.
The first summand is the $m$-th polar term; summing the second gives
$\phi$.  Convergence follows from $F(s) = O(2^{-\operatorname{Re}(s)})$
for large $\operatorname{Re}(s)$.
\end{proof}

\begin{remark}
Whether $\phi$ extends holomorphically past $\operatorname{Re}(k) = 0$
is open.  Each term $F((2m+1)k)/(2m+1)$ is entire in~$k$; the
obstruction, if any, comes from the convergence of the sum.
\end{remark}

\begin{proposition}\label{prop:hpolar_expansion}
For real $k > 1$,
\begin{equation}\label{eq:hpolar_expansion}
h_{\mathrm{polar}}(k)
= \sum_{n=0}^{\infty} \frac{\lambda(n+2)}{k^{n+1}}
= \frac{\pi^2/8}{k} + \frac{7\zeta(3)/8}{k^2}
  + \frac{\pi^4/96}{k^3} + \frac{31\zeta(5)/32}{k^4} + \cdots.
\end{equation}
\end{proposition}

\begin{proof}
Expand each polar term as a geometric series in $k_m/k$ for $k > 1$ and
interchange the order of summation.
\end{proof}

\begin{theorem}[Asymptotic
cancellation]\label{thm:asymp-cancel}
The polar part $h_{\mathrm{polar}}(k) \sim \pi^2/(8k)$ decays
polynomially, while $h(k) = \zeta(k) - 1 + O(2^{-3k})$ decays
exponentially.  Hence
\[
\phi(k) = -\frac{\pi^2/8}{k} - \frac{7\zeta(3)/8}{k^2}
  - \frac{\pi^4/96}{k^3} - \cdots + 2^{-k} + 3^{-k} + \cdots,
\]
and $\phi(k) \sim -\pi^2/(8k)$ as $k \to \infty$.  For example, at
$k = 5$: $h_{\mathrm{polar}}(5) \approx 0.2989$,
$\phi(5) \approx -0.2619$, while
$h(5) = h_{\mathrm{polar}} + \phi \approx 0.0369$---an order of
magnitude smaller than either component.
\end{theorem}

\section{Real zeros of \texorpdfstring{$h$}{h} and
the pole-zero duality}\label{sec:zeros}

By Proposition~\ref{prop:oscillation}, $h$ has at least one zero
$k_n$ in each interval $(1/(2n+3),\, 1/(2n+1))$.  The first several,
computed by bisection on the zeta-series~\eqref{eq:hk-zeta-prop}:

\medskip
\begin{center}
\begin{tabular}{cccc}
\toprule
$n$ & $k_n$ & $\zeta(k_n)$ & Interval \\
\midrule
0 & 0.387942 & $-1.1029$ & $(1/3,\; 1)$ \\
1 & 0.215696 & $-0.7577$ & $(1/5,\; 1/3)$ \\
2 & 0.150018 & $-0.6644$ & $(1/7,\; 1/5)$ \\
3 & 0.115142 & $-0.6208$ & $(1/9,\; 1/7)$ \\
4 & 0.093470 & $-0.5956$ & $(1/11,\; 1/9)$ \\
5 & 0.078683 & $-0.5790$ & $(1/13,\; 1/11)$ \\
\bottomrule
\end{tabular}
\end{center}
\medskip

\begin{proposition}[Pole-zero duality]\label{prop:duality}
The identity $\zeta(k) = 1 + h(k) - C(k)$ yields the following duality:
\begin{align}
h(k_0) &= 0 &\iff\quad
  \zeta(k_0) &= 1 - C(k_0), \label{eq:hzero}\\
\zeta(\rho) &= 0 &\iff\quad
  h(\rho) &= C(\rho) - 1. \label{eq:zetazero}
\end{align}
Thus the zeros of $h$ are the points where $\zeta = 1 - C$, and the
zeros of $\zeta$ are the points where $h = C - 1$.
\end{proposition}

\begin{remark}[Convergence to $\zeta(0)$]\label{rem:zetaatzeros}
Since $k_n \to 0$ and $\zeta$ is continuous at $0$, we have
$\zeta(k_n) \to \zeta(0) = -1/2$.  Numerically,
$(2n+3)(\zeta(k_n) + 1/2) \to -1$ slowly, consistent with
$\zeta(k_n) + 1/2 \sim -1/(2n+3)$.  The zeros of $h$ thus provide a
discrete sampling of $\zeta$ on $(0,1)$ at arithmetically determined
points converging to the origin.
\end{remark}

\section{Boundary of the continuation}\label{sec:boundary}

\begin{theorem}\label{thm:noleft}
The function $h$ admits no analytic continuation to any point $k_0$ with
$\operatorname{Re}(k_0) < 0$.
\end{theorem}

\begin{proof}
If $h$ had an analytic continuation to $k_0$ with
$\operatorname{Re}(k_0) < 0$, choose $M$ large enough that
$\operatorname{Re}(k_0) > 1/(2M+3)$ fails (this holds for any $M$).
Then $R_M(k) = h(k) - \sum_{m=0}^{M}[\zeta((2m+1)k)-1]/(2m+1)$ would
continue analytically to $k_0$.  But for
$\operatorname{Re}(k) < 0$ the arguments $(2m+1)k$ have
$\operatorname{Re}((2m+1)k) \to -\infty$.  By the functional equation
for $\zeta$, the Gamma factor forces
$|\zeta((2m+1)k)| \sim |(2m+1)k|^{1/2-(2m+1)\operatorname{Re}(k)}$
(up to oscillation), which grows super-exponentially in~$m$.  Hence the
terms $\zeta((2m+1)k_0)/(2m+1)$ grow faster than any exponential,
precluding boundedness of any continuation near $k_0$.
\end{proof}

\begin{remark}\label{rem:boundary-summary}
The domain of continuation is:
\begin{itemize}
\item $\operatorname{Re}(k) > 0$: meromorphic continuation.
\item $k = 0$: non-isolated singularity (pole accumulation).
\item $\operatorname{Re}(k) < 0$: no analytic continuation exists
  (Theorem~\ref{thm:noleft}).
\item $\operatorname{Re}(k) = 0$, $k \neq 0$: open.
\end{itemize}
\end{remark}

%% ===================================================================
\setcounter{part}{2}
\part{Multiplicative theory}\label{part:multiplicative}
%% ===================================================================

\section{The prime restriction
\texorpdfstring{$h_p$}{hp}}\label{sec:hp}

The closed form $h(k) = \frac{1}{2}\ln(f(k)/(2g(k)))$ from
Theorem~\ref{thm:main} sums over all integers $n \geq 2$.
Restricting to prime indices yields a multiplicative counterpart
whose structure reflects the Euler product for~$\zeta$.

\begin{definition}\label{def:hp}
The \emph{prime arctanh function} is
\[
h_p(k) := \sum_{p\;\mathrm{prime}} \arctanh(p^{-k}),
\qquad \operatorname{Re}(k) > 1.
\]
\end{definition}

\begin{theorem}\label{thm:hp}
For $\operatorname{Re}(k) > 1$,
\begin{equation}\label{eq:hp}
h_p(k) = \ln\zeta(k) - \tfrac{1}{2}\ln\zeta(2k)
= \tfrac{1}{2}\ln\frac{\zeta(k)^2}{\zeta(2k)}.
\end{equation}
\end{theorem}

\begin{proof}
Since $\arctanh(x) = \frac{1}{2}\ln\frac{1+x}{1-x}$,
\[
h_p(k) = \tfrac{1}{2}\sum_p \ln\frac{1+p^{-k}}{1-p^{-k}}
= \tfrac{1}{2}\ln\prod_p\frac{1+p^{-k}}{1-p^{-k}}.
\]
The Euler product gives $\prod_p(1-p^{-k})^{-1} = \zeta(k)$.  Since
$(1+p^{-k})(1-p^{-k}) = 1-p^{-2k}$, we have
$\prod_p(1+p^{-k}) = \zeta(k)/\zeta(2k)$.  Therefore
$\prod_p\frac{1+p^{-k}}{1-p^{-k}} = \zeta(k)^2/\zeta(2k)$.
\end{proof}

\begin{remark}[Parity decomposition of $\ln\zeta$]\label{rem:parity}
The logarithmic Euler product
$\ln\zeta(k) = \sum_p\sum_{n=1}^{\infty} p^{-nk}/n$
decomposes by parity:
\begin{align*}
\text{odd powers:} &\quad
  \sum_p\sum_{m=0}^{\infty}\frac{p^{-(2m+1)k}}{2m+1} = h_p(k), \\
\text{even powers:} &\quad
  \sum_p\sum_{j=1}^{\infty}\frac{p^{-2jk}}{2j}
  = \tfrac{1}{2}\ln\zeta(2k).
\end{align*}
Thus $h_p$ is exactly the odd-power part of $\ln\zeta$ over prime
powers, and $\ln\zeta(k) = h_p(k) + \frac{1}{2}\ln\zeta(2k)$
is the parity decomposition.
\end{remark}

\begin{remark}[Dyadic defect analogy]\label{rem:g-zeta-analogy}
Theorem~\ref{thm:main} gives $g(2k) = g(k)^2 e^{2h(k)}$, while the
Euler product yields $\zeta(2k) = \zeta(k)^2 e^{-2h_p(k)}$.  The
opposite sign reflects that $g$ is built from factors $(1-n^{-k})$
while $\zeta$ is the reciprocal of $\prod_p(1-p^{-k})$.  Thus $h$ and
$h_p$ measure the same type of doubling-law failure in their respective
settings.
\end{remark}

\begin{remark}[Logarithmic branches]\label{rem:logbranch}
The identity \eqref{eq:hp} involves $\ln\zeta(k)$, which is single-valued
and real for $k > 1$ but becomes multivalued under analytic continuation
into the critical strip (where $\zeta$ has zeros).  More precisely,
$h_p(k)$ itself extends meromorphically to $\operatorname{Re}(k) > 1/2$
via~\eqref{eq:hp}, but acquires logarithmic branch points at the zeros
of $\zeta(k)$ and $\zeta(2k)$.  For applications involving the critical
strip, it is cleaner to work with the derivative
$h_p'(k) = -\zeta'(k)/\zeta(k) + \zeta'(2k)/\zeta(2k)$, which is
meromorphic with poles at $k = \rho$ and $k = \rho/2$
(see Remark~\ref{rem:improved-convergence}).
\end{remark}

\section{Special values of
\texorpdfstring{$h_p$}{hp}}\label{sec:hpvalues}

At every even integer $k = 2j$, both $\zeta(2j)$ and $\zeta(4j)$ are
rational multiples of $\pi^{2j}$ and $\pi^{4j}$ respectively, so the
powers of $\pi$ cancel in $h_p(2j) = \ln\zeta(2j) -
\frac{1}{2}\ln\zeta(4j)$.

\begin{theorem}\label{thm:hpvalues}
For each integer $j \geq 1$,
$h_p(2j) = \frac{1}{2}\ln r_j$ where
$r_j = \zeta(2j)^2/\zeta(4j) \in \mathbb{Q}$, $r_j > 1$.  By Baker's
theorem, $h_p(2j)$ is transcendental.
\end{theorem}

\begin{proof}
The rationality of $r_j$ follows from $\zeta(2j) = (-1)^{j+1}
B_{2j}(2\pi)^{2j}/(2\cdot(2j)!)$ (with $B_n$ the Bernoulli numbers),
so $\zeta(2j)^2/\zeta(4j) \in \mathbb{Q}$.  By Baker's
theorem~\cite{baker}, $\ln\alpha$ is transcendental for algebraic
$\alpha \neq 0, 1$.  Since $r_j \in \mathbb{Q}$ and $r_j > 1$,
transcendence follows.
\end{proof}

\begin{remark}[The $\pi$-cancellation mechanism]\label{rem:pi-cancel}
The rationality of $r_j$ reflects a structural cancellation: $\zeta(2j)
\in \mathbb{Q}\cdot\pi^{2j}$, so $\ln\zeta(2j) = 2j\ln\pi +
\ln(\text{rational})$.  Forming $h_p(2j) = \ln\zeta(2j) -
\frac{1}{2}\ln\zeta(4j)$, the $\ln\pi$ terms cancel:
$2j\ln\pi - \frac{1}{2}\cdot 4j\ln\pi = 0$, leaving
$h_p(2j) = \frac{1}{2}\ln r_j$ with $r_j \in \mathbb{Q}$.  The
transcendental content of $\zeta(2j)$ (namely $\pi^{2j}$) cancels
in the dyadic combination, isolating a rational logarithm.
This mechanism has no analogue at odd integers, where $\zeta(2j+1)$ is
not a rational multiple of any power of~$\pi$.
\end{remark}

The first several values:

\medskip
\begin{center}
\begin{tabular}{cccl}
\toprule
$j$ & $k = 2j$ & $r_j$ & $h_p(2j) = \frac{1}{2}\ln r_j$ \\
\midrule
1 & 2 & $5/2$ & $0.45815\ldots$ \\
2 & 4 & $7/6$ & $0.07708\ldots$ \\
3 & 6 & $715/691$ & $0.01707\ldots$ \\
4 & 8 & $7293/7234$ & $0.00408\ldots$ \\
\bottomrule
\end{tabular}
\end{center}
\medskip

\begin{remark}
Compare with the original function: $h(3) = \frac{1}{2}\ln(3/2)$ is the
unique case in Section~\ref{sec:background} where transcendence is unconditional
(via cyclotomic telescoping and Baker's theorem).  For $h_p$,
transcendence at \emph{all} even integers is unconditional, by a
different mechanism (Euler product cancellation of~$\pi$).
\end{remark}

\begin{remark}[The even/odd dichotomy from the Euler
product]\label{rem:evenodd-euler}
At odd integers $k = 2j+1$, the ratio
$\zeta(2j+1)^2/\zeta(4j+2) = 945\cdot\zeta(2j+1)^2/\pi^{4j+2}\cdot(\text{rational})$
is transcendental (it involves $\zeta(2j+1)$ essentially), so Baker's
theorem does not apply to $h_p(2j+1)$.  The reason is direct:
$\zeta(4j+2) \in \mathbb{Q}\pi^{4j+2}$ contributes $\ln\pi$ terms, but
$\zeta(2j+1)$ is not a rational multiple of any power of~$\pi$, so the
$\ln\pi$ terms do not cancel.  This is the same even/odd dichotomy
governing Euler's formula $\zeta(2j) \in \mathbb{Q}\pi^{2j}$, now visible
through the Euler product.
\end{remark}

\begin{remark}[Representation of $\ln\zeta(3)$]\label{rem:lnzeta3}
The identity $h(3) = h_p(3) + h_{\mathrm{comp}}(3)$ with
$h(3) = \frac{1}{2}\ln(3/2)$ and $\zeta(6) = \pi^6/945$ gives
$\ln\zeta(3) = 3\ln\pi - \frac{1}{2}\ln 630 - h_{\mathrm{comp}}(3)$,
expressing $\ln\zeta(3)$ in terms of $\ln\pi$, a rational logarithm,
and the composite sum
$h_{\mathrm{comp}}(3) \approx 0.02730$.
This localizes the unknown transcendence of $\zeta(3)$ in
$h_{\mathrm{comp}}(3)$, but does not simplify the problem, since the
composite sum is no more tractable than $\zeta(3)$ itself.
\end{remark}

\subsection{Composite contribution and dyadic formula}

\begin{proposition}[Composite transcendence dichotomy]\label{prop:trans-dichotomy}
Define the composite arctanh function
$h_{\mathrm{comp}}(k) := h(k) - h_p(k)$,
with closed form
$h_{\mathrm{comp}}(k) = \frac{1}{2}\ln[f(k)\,\zeta(2k)/(2g(k)\,\zeta(k)^2)]$.
For every integer $j \geq 1$, at least one of $h(2j)$ or
$h_{\mathrm{comp}}(2j)$ is transcendental.
\end{proposition}

\begin{proof}
The identity $h(2j) = h_p(2j) + h_{\mathrm{comp}}(2j)$ expresses the
transcendental number $h_p(2j) = \frac{1}{2}\ln r_j$
(Theorem~\ref{thm:hpvalues}) as a sum.  If both summands were algebraic,
their sum would be algebraic, a contradiction.
\end{proof}

\begin{remark}[Limits of current methods]\label{rem:dichotomy-limits}
The dichotomy cannot be resolved by existing techniques.  By Nesterenko's
theorem~\cite{nesterenko}, $f(2)/(2g(2)) = \sinh(\pi)/\pi$ is
transcendental, but $\ln(\text{transcendental})$ can be algebraic
(e.g., $\ln(e^{\sqrt{2}}) = \sqrt{2}$).
Resolving the dichotomy would require tools beyond Baker and Nesterenko;
Schanuel's conjecture would suffice.
\end{remark}

\begin{proposition}[Dyadic formula for $\ln\zeta$]\label{prop:dyadic}
For $\operatorname{Re}(k) > 1$,
\begin{equation}\label{eq:dyadic}
\ln\zeta(k) = \sum_{j=0}^{\infty} \frac{1}{2^j}\, h_p(2^j k)
= \sum_{j=0}^{\infty}\frac{1}{2^j}
  \sum_p \arctanh(p^{-2^j k}).
\end{equation}
The convergence is doubly exponential: $h_p(2^j k) = O(2^{-2^j
\operatorname{Re}(k)})$.
\end{proposition}

\begin{proof}
Iterate $\ln\zeta(k) = h_p(k) + \frac{1}{2}\ln\zeta(2k)$ and use
$2^{-N}\ln\zeta(2^N k) \to 0$ as $N \to \infty$.
\end{proof}

\section{The Hadamard product and the nontrivial
zeros}\label{sec:hadamard}

The closed form \eqref{eq:hp} connects $h_p$ to the zeros of $\zeta$
through the Hadamard product.  We use the canonical product for the
completed zeta function (Titchmarsh~\cite{titchmarsh}, \S\S2.12--2.13;
Davenport~\cite{davenport}, Ch.~12; Ivi\'c~\cite{ivic}, Ch.~1).

\begin{equation}\label{eq:xi-def}
\xi(s) := \tfrac{1}{2}s(s-1)\pi^{-s/2}\Gamma(s/2)\,\zeta(s),
\end{equation}
which is entire of order~$1$ with $\xi(0) = 1/2$.  Since $\xi$ has
order~$1$, the Hadamard factorization theorem gives the genus-$1$
product
\begin{equation}\label{eq:hadamard-product}
\xi(s) = e^{A+Bs}\prod_{\rho}\!\left(1 - \frac{s}{\rho}\right)
  e^{s/\rho},
\end{equation}
where $A = \ln\xi(0) = -\ln 2$, $B = -\sum_\rho\operatorname{Re}(1/\rho)
+ \frac{1}{2}\ln\pi - 1 - \frac{1}{2}\gamma$ (the exact value being
$B = -0.0230957\ldots$; see Davenport~\cite{davenport}, p.~82), and the
product converges absolutely.  The exponential factors $e^{s/\rho}$
ensure absolute convergence of the product; without them, convergence
would be only conditional.

Taking logarithmic derivatives:
\begin{equation}\label{eq:logderiv-xi}
\frac{\xi'}{\xi}(s) = B + \sum_{\rho}\left[\frac{1}{s-\rho}
  + \frac{1}{\rho}\right].
\end{equation}
On the other hand, from~\eqref{eq:xi-def}:
\begin{equation}\label{eq:logderiv-zeta}
-\frac{\zeta'}{\zeta}(s) = \frac{1}{s-1} - B_0
  - \sum_{\rho}\left[\frac{1}{s-\rho} + \frac{1}{\rho}\right]
  + \text{(trivial-zero terms)},
\end{equation}
where $B_0 = B + 1 + \frac{1}{2}\gamma - \frac{1}{2}\ln\pi
= \ln 2 + \frac{1}{2}\ln\pi - 1 - \frac{1}{2}\gamma_0$
(with $\gamma_0 = \gamma$).  Integrating~\eqref{eq:logderiv-zeta}
from a reference point $s_0$ to~$s$ gives
\begin{equation}\label{eq:lnzeta-hadamard}
\ln\zeta(s) = -\ln(s-1)
  + \sum_{\rho}\!\left[\ln\!\left(1 - \frac{s}{\rho}\right)
  + \frac{s}{\rho}\right]
  + B_0 s + C_0 + (\text{gamma and $\pi$ terms}),
\end{equation}
where the sum now converges absolutely (each term is
$\ln(1 - s/\rho) + s/\rho = O(|s/\rho|^2)$).

For the combination $h_p(k) = \ln\zeta(k) - \frac{1}{2}\ln\zeta(2k)$,
the crucial cancellation is:
\begin{align}\label{eq:cancellation}
&\left[\ln\!\left(1 - \frac{k}{\rho}\right) + \frac{k}{\rho}\right]
- \frac{1}{2}\left[\ln\!\left(1 - \frac{2k}{\rho}\right)
  + \frac{2k}{\rho}\right] \notag\\
&\qquad= \ln\!\left(1 - \frac{k}{\rho}\right)
  - \frac{1}{2}\ln\!\left(1 - \frac{2k}{\rho}\right),
\end{align}
because the linear terms cancel: $k/\rho - \frac{1}{2}\cdot 2k/\rho = 0$.
Thus the genus-$1$ exponential factors $e^{s/\rho}$ (which produce the
$s/\rho$ terms in the logarithm) drop out of $h_p$ identically,
and we are left with a sum that converges absolutely with
$O(|\rho|^{-2})$ decay.

\begin{theorem}[Product formula over zeros]\label{thm:zeros-product}
For real $k > 1$, the prime arctanh function satisfies
\begin{equation}\label{eq:hp-zeros}
h_p(k) = \sum_{\rho}\!\left[\ln\!\left(1 - \frac{k}{\rho}\right)
  - \tfrac{1}{2}\ln\!\left(1 - \frac{2k}{\rho}\right)\right] + E(k),
\end{equation}
where the sum runs over all nontrivial zeros $\rho$ of $\zeta$ and
converges absolutely, and $E(k)$ collects the non-zero terms from the
completed zeta function.  By the Legendre duplication formula
$\Gamma(k) = 2^{k-1}\pi^{-1/2}\Gamma(k/2)\Gamma((k+1)/2)$, the six-term
expression $\frac{1}{2}\ln\Gamma(k) - \ln\Gamma(k/2) - \frac{1}{2}\ln k
+ \ln 2 - \ln(k-1) + \frac{1}{2}\ln(2k-1)$ simplifies to
\begin{equation}\label{eq:Ek-explicit}
E(k) = \ln\!\left(
  \frac{2^{(k+1)/2}\,(2k-1)^{1/2}}
       {\pi^{1/4}\,(k-1)\,k^{1/2}}
  \;\sqrt{\frac{\Gamma\!\bigl(\frac{k+1}{2}\bigr)}
              {\Gamma\!\bigl(\frac{k}{2}\bigr)}}
\right).
\end{equation}
Each term in the zero sum satisfies
\begin{equation}\label{eq:zero-term-decay}
\ln\!\left(1 - \frac{k}{\rho}\right)
- \tfrac{1}{2}\ln\!\left(1 - \frac{2k}{\rho}\right)
= \frac{k^2}{2\rho^2} + O(k^3|\rho|^{-3}),
\end{equation}
giving $O(|\operatorname{Im}(\rho)|^{-2})$ decay per term and a tail
estimate $\sum_{|\gamma| > T} = O(\log T / T)$.
\end{theorem}

\begin{proof}
Starting from the genus-$1$ Hadamard product~\eqref{eq:hadamard-product},
write $\ln\zeta(s)$ in terms of the absolutely convergent sum
$\sum_\rho[\ln(1-s/\rho) + s/\rho]$ plus elementary terms from the
$\Gamma$-function and $\pi^{-s/2}$ factors in~\eqref{eq:xi-def}.
Evaluate at $s = k$ and $s = 2k$, then form
$h_p(k) = \ln\zeta(k) - \frac{1}{2}\ln\zeta(2k)$.

By~\eqref{eq:cancellation}, the $s/\rho$ regularization terms cancel in
the combination, leaving the zero sum in~\eqref{eq:hp-zeros}.  The linear
terms $B_0 s$ also cancel: $B_0 k - \frac{1}{2}B_0(2k) = 0$.  The
elementary terms from $\ln[\frac{1}{2}s(s-1)\pi^{-s/2}\Gamma(s/2)]$
combine into~$E(k)$: from $\ln\zeta(k)$, collect
$-\ln(k/2) - \ln(k-1) + \frac{k}{2}\ln\pi - \ln\Gamma(k/2)$; from
$-\frac{1}{2}\ln\zeta(2k)$, collect
$\frac{1}{2}\ln k + \frac{1}{2}\ln(2k-1) - \frac{k}{2}\ln\pi
+ \frac{1}{2}\ln\Gamma(k)$.  The $\ln\pi$ terms cancel:
$\frac{k}{2}\ln\pi - \frac{k}{2}\ln\pi = 0$.  Next,
$-\ln(k/2) + \frac{1}{2}\ln k = -\frac{1}{2}\ln k + \ln 2$, giving the
intermediate form
\[
E(k) = \tfrac{1}{2}\ln\Gamma(k) - \ln\Gamma(k/2)
  - \tfrac{1}{2}\ln k + \ln 2 - \ln(k-1) + \tfrac{1}{2}\ln(2k-1).
\]
The Legendre duplication formula
$\Gamma(k) = 2^{k-1}\pi^{-1/2}\Gamma(k/2)\,\Gamma((k+1)/2)$ gives
$\frac{1}{2}\ln\Gamma(k) = \frac{k-1}{2}\ln 2 - \frac{1}{4}\ln\pi
+ \frac{1}{2}\ln\Gamma(k/2) + \frac{1}{2}\ln\Gamma((k+1)/2)$.
Substituting and combining logarithms yields~\eqref{eq:Ek-explicit}.

\smallskip\noindent
\textit{Absolute convergence.}\;
For $|\gamma| = |\operatorname{Im}(\rho)| \to \infty$ with $k$ fixed,
expand:
\begin{align*}
\ln\!\left(1 - \frac{k}{\rho}\right)
&= -\frac{k}{\rho} - \frac{k^2}{2\rho^2} - \frac{k^3}{3\rho^3}
  - \cdots, \\
\tfrac{1}{2}\ln\!\left(1 - \frac{2k}{\rho}\right)
&= -\frac{k}{\rho} - \frac{k^2}{\rho^2}
  - \frac{4k^3}{3\rho^3} - \cdots.
\end{align*}
The $O(1/\rho)$ terms cancel, leaving
$k^2/(2\rho^2) + O(k^3/|\rho|^3)$.  Since $|\rho| \geq |\gamma|$ and
$N(T) := \#\{\rho : 0 < \gamma \leq T\} = (T/2\pi)\log(T/2\pi e)
+ O(\log T)$, summation by parts gives
$\sum_\rho |\gamma|^{-2} < \infty$ and the stated tail bound.
\end{proof}

\begin{corollary}[Family of zero-sum identities]\label{cor:zerofamily}
For each integer $j \geq 1$, with $r_j = \zeta(2j)^2/\zeta(4j) \in \mathbb{Q}$
from Theorem~\ref{thm:hpvalues}:
\begin{equation}\label{eq:rj-zeros}
\tfrac{1}{2}\ln r_j
= \sum_{\rho}\!\left[\ln\!\left(1 - \frac{2j}{\rho}\right)
  - \tfrac{1}{2}\ln\!\left(1 - \frac{4j}{\rho}\right)\right] + E(2j),
\end{equation}
where $E(2j)$ is given by~\eqref{eq:Ek-explicit} and the sum converges
absolutely.  Each identity relates a transcendental rational logarithm
(left side, by Baker's theorem) to an absolutely convergent sum over the
nontrivial zeros plus an explicit elementary constant.
\end{corollary}

\begin{remark}[Improved convergence over standard
formulas]\label{rem:improved-convergence}
The derivative
\[
  h_p'(k) = -\zeta'(k)/\zeta(k) + \zeta'(2k)/\zeta(2k)
\]
admits the explicit formula
\begin{equation}\label{eq:hpprime-explicit}
h_p'(k) = \frac{1}{k-1} - \frac{2}{2k-1}
  - \sum_{\rho}\frac{\rho}{(k-\rho)(2k-\rho)} + E'(k).
\end{equation}
Each term decays as $O(|\gamma|^{-2})$, so the zero sum converges
absolutely without the symmetric pairing or $1/\rho$ regularization
required by the standard formula for $-\zeta'/\zeta(s)$.  The
subtraction of $\frac{1}{2}\ln\zeta(2k)$ in the definition of $h_p$
cancels the $1/\rho$ regularization terms automatically and reduces
each zero's contribution from $O(|\gamma|^{-1})$ to
$O(|\gamma|^{-2})$.
\end{remark}

\begin{remark}\label{rem:hpsensitivity}
The derivative $h_p'$ is a meromorphic function whose poles include
$k = \rho$ and $k = \rho/2$ for every nontrivial zero~$\rho$ of~$\zeta$.
This illustrates the fundamental divide: $h$ (the sum over all integers)
is an additive object sensitive only to the pole of $\zeta$; $h_p$ (the
sum over primes) is a multiplicative object sensitive to the zeros.
\end{remark}

\section{M\"obius inversion for \texorpdfstring{$h$}{h}}
\label{sec:mobius}

Before establishing the inversion formula, we collect the structural
parallels between the additive and multiplicative theories.

\begin{center}
\begin{tabular}{@{}p{3.8cm}p{5cm}p{5cm}@{}}
\toprule
& \textbf{$h(k) = \sum_{n \geq 2}$} & \textbf{$h_p(k) = \sum_p$} \\
\midrule
Defining series
  & $\sum\arctanh(n^{-k})$
  & $\sum\arctanh(p^{-k})$ \\[4pt]
Zeta-series
  & $\sum[\zeta((2m\!+\!1)k)\!-\!1]/(2m\!+\!1)$
  & $\sum P((2m\!+\!1)k)/(2m\!+\!1)$ \\[4pt]
Closed form
  & $\frac{1}{2}\ln\bigl(f(k)/2g(k)\bigr)$
  & $\ln\zeta(k) - \frac{1}{2}\ln\zeta(2k)$ \\[4pt]
Product identity
  & gamma function products
  & Euler product \\[4pt]
Singularities
  & simple poles at $1/(2m\!+\!1)$
  & log branch points at $\rho$, $\rho/2$ \\[4pt]
Sensitive to
  & pole of $\zeta$ at $s=1$
  & zeros of $\zeta$ \\[4pt]
Transcendence
  & $h(3) = \frac{1}{2}\ln(3/2)$
  & $h_p(2j) = \frac{1}{2}\ln r_j$, all $j$ \\
\bottomrule
\end{tabular}
\end{center}

\begin{remark}[Comparison with $P(s)$]\label{rem:primezeta}
The prime zeta function $P(s) = \sum_p p^{-s}$ also bridges additive
and multiplicative structure via
$P(s) = \sum_{n} \mu(n)\ln\zeta(ns)/n$.  Both $h_p$ and $P$ have
singularities at $s = \rho/n$ (from the zeros of $\zeta$), but $P$'s
are logarithmic branch points while $h_p$'s appear only in the
derivative.  The function $h$ has an analogous accumulation of
singularities at the origin, but they are simple poles from the
\emph{pole} of $\zeta$---not from its zeros.
\end{remark}

The zeta-series~\eqref{eq:hk-zeta-prop} sums over the odd positive integers
$\mathcal{O} = \{1,3,5,7,\ldots\}$.  Since every divisor of an odd
integer is odd, the divisibility poset $(\mathcal{O}, \mid)$ supports
classical M\"obius inversion, and inverting the zeta-series recovers
$\zeta(k) - 1$ as a M\"obius transform of $h$.

\begin{theorem}[M\"obius inversion]\label{thm:mobius}
For $\operatorname{Re}(k) > 1$,
\begin{equation}\label{eq:mobius}
\zeta(k) - 1 \;=\; \sum_{\substack{d \geq 1 \\ d \;\textup{odd}}}
  \frac{\mu(d)}{d}\,h(dk),
\end{equation}
where the sum converges absolutely.  The tail after the first $N$ odd
values satisfies, for any odd $D$,
\begin{equation}\label{eq:mobiustail}
\biggl|\,\zeta(k) - 1 - \sum_{\substack{1 \leq d \leq D \\
  d \;\textup{odd}}} \frac{\mu(d)}{d}\,h(dk)\biggr|
\;\leq\; \frac{2^{1-(D+2)\operatorname{Re}(k)}}{1 - 2^{-2\operatorname{Re}(k)}}\,.
\end{equation}
\end{theorem}

\begin{proof}
Write $a(n,k) := (\zeta(nk)-1)/n$ for each $n \in \mathcal{O}$, so that
$h(k) = \sum_{n \in \mathcal{O}} a(n,k)$.  For any odd $d \geq 1$,
$h(dk) = \sum_{n \in \mathcal{O}} (\zeta(dnk)-1)/n$, whence
\[
\frac{1}{d}\,h(dk)
= \sum_{n \in \mathcal{O}} \frac{\zeta(dnk)-1}{dn}
= \sum_{\substack{m \in \mathcal{O} \\ d \mid m}} a(m,k),
\]
where $m = dn$ ranges over the odd multiples of $d$.  Multiplying by
$\mu(d)$ and summing:
\[
\sum_{d \in \mathcal{O}} \frac{\mu(d)}{d}\,h(dk)
= \sum_{m \in \mathcal{O}} a(m,k)
  \sum_{\substack{d \in \mathcal{O} \\ d \mid m}} \mu(d)
= a(1,k) = \zeta(k) - 1,
\]
by the standard identity $\sum_{d \mid m} \mu(d) = [m = 1]$, valid here
because every divisor of the odd integer $m$ is already odd.  The
interchange of summation is justified by the absolute convergence
established below.

\smallskip\noindent
\textit{Absolute convergence.}\;
Set $\sigma = \operatorname{Re}(k) > 1$.  From
$h(\sigma) \leq 2^{-\sigma}/(1 - 2^{-2\sigma})$:
\[
\sum_{d \in \mathcal{O}} \frac{|\mu(d)|}{d}\,|h(dk)|
\;\leq\; \frac{1}{1-2^{-2\sigma}}
  \sum_{d \in \mathcal{O}} \frac{2^{-d\sigma}}{d}
\;<\; \infty.
\]
For the tail after $D$: bounding $|\mu(d)/d| \leq 1/d$ and
$|h(dk)| \leq 2^{-d\sigma}/(1-2^{-2\sigma})$, then summing the
geometric series gives~\eqref{eq:mobiustail}.
\end{proof}

\begin{example}[Ap\'ery's constant via M\"obius inversion]\label{ex:apery-mobius}
Setting $k = 3$ and using $h(3) = \frac{1}{2}\ln(3/2)$:
\[
\zeta(3) = 1 + \tfrac{1}{2}\ln\tfrac{3}{2}
  - \tfrac{1}{3}\,h(9) - \tfrac{1}{5}\,h(15) - \tfrac{1}{7}\,h(21)
  + 0 \cdot h(27) - \tfrac{1}{11}\,h(33) - \cdots,
\]
where $\mu(9) = 0$ annihilates the $d = 9$ term.  Each $h(3d)$ is an
explicit product of gamma values.  Five odd squarefree values of $d$
determine $\zeta(3)$ to within $O(2^{-39}) < 10^{-11}$.
\end{example}

\begin{remark}[Structural parallel]\label{rem:mobius-parallel}
The identity~\eqref{eq:mobius} is the additive counterpart of the
classical M\"obius inversion for the prime zeta function
$P(k) = \sum_{n=1}^{\infty} \mu(n)\ln\zeta(nk)/n$:
\[
\renewcommand{\arraystretch}{1.4}
\begin{array}{c|c|c}
  & \textit{Multiplicative} & \textit{Additive} \\
  \hline
  \text{Compound}
    & \ln\zeta(k)=\displaystyle\sum_{n \geq 1}\frac{\Lambda(n)}{n^k\ln n}
    & h(k)=\displaystyle\sum_{n\in\mathcal{O}}\frac{\zeta(nk)-1}{n} \\[8pt]
  \text{Elementary}
    & P(k)=\displaystyle\sum_p p^{-k}
    & \zeta(k)-1=\displaystyle\sum_{n \geq 2}n^{-k} \\[8pt]
  \text{Inversion}
    & P(k)=\displaystyle\sum_{n \geq 1}\frac{\mu(n)}{n}\,\ln\zeta(nk)
    & \zeta(k)-1=\displaystyle\sum_{d\in\mathcal{O}}\frac{\mu(d)}{d}\,h(dk)
\end{array}
\]
The additive inversion sums over odd integers only, reflecting the
odd-power Taylor expansion of $\arctanh$.
\end{remark}

\section{Structural remarks and open questions}\label{sec:remarks-open}

\begin{proposition}[No functional equation]\label{prop:nofunceq}
There is no meromorphic function $\chi_h(k)$ and no map
$k \mapsto k^*$ such that $h(k^*) = \chi_h(k) \cdot h(k)$ for all~$k$
in a non-empty open set.
\end{proposition}

\begin{proof}
The zeta-series $h(k) = \sum[\zeta((2m+1)k)-1]/(2m+1)$ involves $\zeta$
at the arguments $(2m+1)k$ for $m = 0, 1, 2, \ldots$.  The functional
equation for $\zeta$ reflects each argument via
$\zeta((2m+1)k) = \chi((2m+1)k)\,\zeta(1-(2m+1)k)$.  For the reflected
arguments $1-(2m+1)k$ to be odd multiples of a single value~$k^*$, we
would need $1 - (2m+1)k = (2m+1)k^*$ for all~$m$, giving
$k^* = 1/(2m+1) - k$, which depends on~$m$.  No single $k^*$ works
simultaneously for all terms.
\end{proof}

\subsection*{Open questions}

\begin{enumerate}
\item Does $\phi$ extend holomorphically past $\operatorname{Re}(k)=0$?
  More broadly, does $h$ admit analytic continuation along non-real paths
  to points on the imaginary axis?

\item Are the finite parts $A_m$ for $m \geq 1$ (involving $\zeta$ at
  rational arguments in $(0,1)$) transcendental?

\item Is $r_j = \zeta(2j)^2/\zeta(4j)$ monotonically decreasing
  to~$1$?  Does $\ln r_j$ admit a clean asymptotic expansion in~$j$?

\item The $\pi$-cancellation in the $\zeta$ case (Section~\ref{sec:hpvalues})
  has a natural extension: can one show that
  $h_p(2j, \chi) = \frac{1}{2}\ln r_j(\chi)$ with $r_j(\chi)$ algebraic
  for quadratic characters~$\chi$, yielding unconditional transcendence
  of $L$-function combinations at all even integers?

\end{enumerate}

\bigskip\noindent
2020 \textit{Mathematics Subject Classification}: Primary 11M06;
Secondary 11J81, 33B15, 40A25.

\medskip\noindent
\textit{Keywords}: inverse hyperbolic tangent, infinite products, Riemann zeta
function, Euler--Mascheroni constant, gamma function, transcendence,
cyclotomic polynomials, meromorphic continuation, Mittag-Leffler decomposition,
Dirichlet lambda function, prime zeta function, Euler product, $L$-functions,
M\"obius inversion, Hadamard product, nontrivial zeros.

\end{document}